\documentclass[11pt]{article}

\usepackage[margin=1in]{geometry}
\usepackage[T1]{fontenc}
\usepackage[utf8]{inputenc}
\usepackage{lmodern}
\usepackage{microtype}
\usepackage{amsmath,amssymb,amsthm,mathtools}
\usepackage{bm}
\usepackage{enumitem}
\usepackage[hidelinks]{hyperref}
\usepackage[nameinlink,noabbrev]{cleveref}

\hypersetup{
  pdftitle={Risk-Constrained Kelly for Mutually Exclusive Outcomes},
  pdfauthor={Christopher D. Long}
}

\newtheorem{theorem}{Theorem}[section]
\newtheorem{proposition}[theorem]{Proposition}
\newtheorem{lemma}[theorem]{Lemma}
\newtheorem{corollary}[theorem]{Corollary}
\theoremstyle{definition}
\newtheorem{definition}[theorem]{Definition}
\newtheorem{example}[theorem]{Example}
\theoremstyle{remark}
\newtheorem{remark}[theorem]{Remark}

\newcommand{\R}{\mathbb{R}}

\title{Risk-Constrained Kelly for Mutually Exclusive Outcomes:\
CRRA Support Invariance and Logarithmic One-Dimensional Calibration}
\author{Christopher D. Long\\Headlamp Software\\\texttt{galizur@gmail.com}}
\date{}

\begin{document}
\maketitle

\begin{abstract}
We study the finite mutually exclusive outcome version of risk-constrained Kelly optimization with explicit state prices. The market has outcome probabilities $p_i>0$, state prices $q_i>0$, terminal wealths $W_i=c+x_i/q_i$, and a drawdown-surrogate constraint
\[
\sum_{i=1}^n p_i W_i^{-\lambda}\le 1,\qquad \lambda>0.
\]
For constant relative risk aversion utility, we work primarily in the standard overround regime $\sum_i q_i>1$, where every optimizer is necessarily non-full-support. Under the usual unique likelihood-ratio prefix hypothesis for the unconstrained problem, we prove that the constrained optimizer has exactly the same active set. Thus, in the regime where the prefix theorem is meaningful, the risk constraint deforms the funded wealth profile but does not change the active set. The support is therefore invariant across both the CRRA parameter and the drawdown-surrogate parameter.

We then isolate the logarithmic case $\gamma=1$. Once the common active prefix is known, the constrained problem reduces to a one-dimensional outer calibration together with independent one-dimensional inner equations on the active states. In this case we prove existence, uniqueness, and monotonicity for the inner solves, derive a complete calibration theorem, and record the resulting structured algorithm. We treat the fair and subfair regimes only as boundary cases: full-support phenomena can occur there, so the overround prefix theory no longer yields a parallel exact description of comparable sharpness. A numerical example illustrates how the risk constraint alters the funded wealth profile while leaving support unchanged.
\end{abstract}

\section{Introduction}

Risk-constrained Kelly gambling is naturally formulated as a finite-state utility maximization problem with an additional downside-risk constraint. In the general return-vector setting, Busseti, Ryu, and Boyd introduced a tractable convex surrogate for drawdown control and analyzed the resulting growth-risk tradeoff \cite{BussetiRyuBoyd2016}. Their framework is deliberately broad: one optimizes over a generic vector of bet fractions against an arbitrary nonnegative return vector. That breadth comes at a structural price. In the mutually exclusive finite-state market considered here, the state-price/wealth-profile coordinates
\[
W_i=c+\frac{x_i}{q_i},
\qquad
L_i=\frac{p_i}{q_i},
\]
reveal exact support geometry that is invisible in a generic return-vector formulation.

The key technical point is that the likelihood ratios $L_i$ act as sufficient statistics for support selection. In these coordinates, the unconstrained problem exhibits the familiar prefix structure from ordinary Kelly gambling \cite{Kelly1956}, and the constrained first-order system factors statewise through the same likelihood-ratio scores. This makes it possible to prove a theorem that has no direct analogue in the generic convex formulation: in the overround non-full-support regime, the drawdown-surrogate constraint preserves the unconstrained active set and changes only the funded wealth profile. The sharp support results of the paper live entirely in this regime; the fair and subfair cases are discussed separately only as boundary regimes where full support may occur.

This paper is part of a broader state-price/wealth-profile program for gambling and related finite-state allocation problems. In \cite{Long2026Implicit}, the single-event multinomial Kelly problem is solved directly in wealth-profile coordinates via implicit state positions. In \cite{Long2026Support}, support selection and eventwise decoupling are analyzed for simultaneous independent multi-outcome bets. The present paper applies the same methodology to a constrained one-event problem.

Our main contributions are these.
\begin{enumerate}[label=\textnormal{(\roman*)},leftmargin=*]
\item In the standard bookmaker overround regime $\sum_i q_i>1$, and under the usual unique likelihood-ratio prefix condition $L_{k_*}>\tau_{k_*}\ge L_{k_*+1}$, we extend the support analysis from logarithmic utility to the full CRRA family: the constrained optimizer has the same likelihood-ratio prefix support as the unconstrained problem.
\item We isolate the precise mechanism behind support invariance: the active-state first-order equation factors through a statewise map that remains strictly decreasing, so the risk constraint affects only multiplier calibration on a support already determined by the unconstrained problem.
\item We then specialize to logarithmic utility and derive a complete one-dimensional outer calibration theorem, together with existence, uniqueness, and monotonicity of the active-state inner equations. The logarithmic case is isolated because it yields the cleanest explicit calibration theory on the common support.
\item We clarify the role of the fair and subfair regimes. Full support can occur only when $\sum_i q_i\le 1$; in those boundary regimes we record benchmark structural observations, but we do not claim a parallel exact support-and-calibration theory of the same sharpness as in the overround case.
\end{enumerate}

The paper is self-contained. \Cref{sec:model} sets up the model and records a preliminary full-support obstruction. \Cref{sec:unconstrained-crra} proves the unconstrained CRRA prefix theorem in the overround regime. \Cref{sec:constrained-crra} proves the constrained CRRA support-invariance theorem and states a general penalty principle. \Cref{sec:log} specializes to logarithmic utility and develops the one-dimensional calibration theory on the common overround support. \Cref{sec:full-support} records fair and subfair boundary observations, and \cref{sec:example} gives a numerical example.

\section{Model and preliminary observations}
\label{sec:model}

Fix $n\ge 2$. Let $p_i>0$ be subjective probabilities with $\sum_{i=1}^n p_i=1$, and let $q_i>0$ be state prices. One unit of claim $i$ costs $q_i$ and pays $1$ only in outcome $i$.

The decision variables are retained cash $c\ge 0$ and state-claim expenditures $x_i\ge 0$, subject to the budget constraint
\[
c+\sum_{i=1}^n x_i=1.
\]
Terminal wealth in outcome $i$ is
\[
W_i=c+\frac{x_i}{q_i}.
\]

Fix a CRRA parameter $\gamma>0$ and a risk parameter $\lambda>0$. We consider the risk-constrained program
\begin{equation}
\label{eq:primal}
\begin{aligned}
\max_{c,x}\quad & \sum_{i=1}^n p_i U_\gamma\!\left(c+\frac{x_i}{q_i}\right) \\
\text{s.t.}\quad & c+\sum_{i=1}^n x_i=1, \\
& c\ge 0,\quad x_i\ge 0\quad (1\le i\le n), \\
& \sum_{i=1}^n p_i\left(c+\frac{x_i}{q_i}\right)^{-\lambda}\le 1,
\end{aligned}
\end{equation}
where
\[
U_\gamma(w):=
\begin{cases}
\dfrac{w^{1-\gamma}}{1-\gamma}, & \gamma\neq 1,\\[6pt]
\log w, & \gamma=1.
\end{cases}
\]

Because $U_\gamma$ is strictly concave on $(0,\infty)$ and the feasible set is convex, the optimal terminal wealth vector is unique.

\begin{definition}[Likelihood-ratio scores]
Define
\[
L_i:=\frac{p_i}{q_i},\qquad 1\le i\le n,
\]
and relabel indices so that
\[
L_1\ge L_2\ge \cdots \ge L_n.
\]
For $k\in\{1,\dots,n\}$ write
\[
P_k:=\sum_{i=1}^k p_i,
\qquad
Q_k:=\sum_{i=1}^k q_i,
\]
and, whenever $Q_k\neq 1$,
\[
\tau_k:=\frac{1-P_k}{1-Q_k}.
\]
\end{definition}

The first structural observation isolates the regime in which full support is even possible.

\begin{proposition}[Full support requires a fair or subfair market]
\label{prop:full-support-obstruction}
Consider either the unconstrained CRRA problem obtained by removing the risk constraint from \eqref{eq:primal}, or the constrained problem \eqref{eq:primal}. If an optimizer satisfies $x_i>0$ for every $i$, then
\[
\sum_{i=1}^n q_i\le 1.
\]
If, in addition, $c>0$, then in fact $\sum_{i=1}^n q_i=1$.

In particular, in the overround regime
\[
Q_n:=\sum_{i=1}^n q_i>1,
\]
no optimizer can have full claim support.
\end{proposition}

\begin{proof}
The constrained and unconstrained cases are proved identically; the latter is obtained by setting the multiplier of the risk constraint equal to zero.

Let $\eta\ge 0$ denote the multiplier of the risk constraint, with $\eta=0$ in the unconstrained case. If $x_i>0$ for all $i$, then the stationarity condition with respect to $x_i$ gives
\[
\frac{p_i}{q_i}W_i^{-\gamma}+\eta\lambda \frac{p_i}{q_i}W_i^{-(\lambda+1)}=\nu,
\qquad 1\le i\le n.
\]
Multiplying by $q_i$ and summing yields
\begin{equation}
\label{eq:sum-active-full}
\sum_{i=1}^n p_i\bigl(W_i^{-\gamma}+\eta\lambda W_i^{-(\lambda+1)}\bigr)=\nu\sum_{i=1}^n q_i.
\end{equation}
On the other hand, stationarity with respect to $c$ gives
\[
\sum_{i=1}^n p_i\bigl(W_i^{-\gamma}+\eta\lambda W_i^{-(\lambda+1)}\bigr)-\nu+\rho=0,
\]
where $\rho\ge 0$ is the multiplier of the cash nonnegativity constraint. Hence
\begin{equation}
\label{eq:sum-cash-full}
\sum_{i=1}^n p_i\bigl(W_i^{-\gamma}+\eta\lambda W_i^{-(\lambda+1)}\bigr)=\nu-\rho.
\end{equation}
Comparing \eqref{eq:sum-active-full} and \eqref{eq:sum-cash-full} gives
\[
\nu\sum_{i=1}^n q_i=\nu-\rho\le \nu.
\]
Since $\nu>0$ by the active-state first-order condition, we conclude that $\sum_i q_i\le 1$. If $c>0$, then complementary slackness gives $\rho=0$, so equality holds.
\end{proof}

\begin{corollary}[Overround reduction to the non-full-support regime]
\label{cor:overround}
Assume $Q_n>1$. Then every optimizer, constrained or unconstrained, has a nonempty inactive set and strictly positive cash level.
\end{corollary}

\begin{proof}
The first statement follows immediately from \cref{prop:full-support-obstruction}. If some state is inactive, then $x_j=0$ for at least one $j$, and therefore $W_j=c>0$.
\end{proof}

\begin{remark}[Constraint qualification in the main overround regime]
\label{rem:slater}
In the overround/prefix regime treated in the main support and calibration theorems, Slater's condition holds for \eqref{eq:primal}. The no-bet point $(c,x)=(1,0)$ satisfies the risk constraint with equality. If $L_i>1$ for some state $i$ (and in particular for $i=1$ under \eqref{eq:prefix-condition}), then for sufficiently small $\varepsilon>0$ the perturbation
\[
c=1-\varepsilon,
\qquad
x_i=\varepsilon,
\qquad
x_j=0\quad (j\ne i)
\]
is strictly feasible. Indeed, the corresponding terminal wealths satisfy
\[
W_i=1+\varepsilon\Bigl(\frac{1}{q_i}-1\Bigr),
\qquad
W_j=1-\varepsilon\quad (j\ne i),
\]
and a first-order expansion gives
\[
\sum_{j=1}^n p_jW_j^{-\lambda}
=1-\lambda(L_i-1)\varepsilon+O(\varepsilon^2)<1.
\]
Thus the convex feasible set has a strictly feasible point throughout the regime relevant to the KKT analysis below.
\end{remark}

\section{Unconstrained CRRA support in the overround regime}
\label{sec:unconstrained-crra}

We now study the unconstrained problem obtained from \eqref{eq:primal} by removing the risk constraint. Throughout this section we assume the overround condition $Q_n>1$, so that the optimizer is automatically non-full-support by \cref{cor:overround}.

\begin{proposition}[CRRA prefix theorem]
\label{prop:crra-prefix}
Assume $Q_n>1$ and suppose there exists a unique index
\[
k_*\in\{1,\dots,n-1\}
\]
such that
\begin{equation}
\label{eq:prefix-condition}
L_{k_*}>\tau_{k_*}\ge L_{k_*+1}.
\end{equation}
Then, for every $\gamma>0$, the unconstrained CRRA optimizer has active set
\[
A_*:=\{1,\dots,k_*\}.
\]
In particular, the unconstrained active set is independent of $\gamma$.

More precisely, the optimal terminal wealth vector has the form
\[
W_i^{\mathrm{ucrra}}=
\begin{cases}
c_*\bigl(L_i/\tau_*\bigr)^{1/\gamma}, & i\le k_*,\\[2mm]
c_*, & i>k_*,
\end{cases}
\qquad
\tau_*:=\tau_{k_*},
\]
where
\[
c_*=
\frac{1}{(1-Q_*)+\sum_{i\in A_*} q_i(L_i/\tau_*)^{1/\gamma}},
\qquad
Q_*:=Q_{k_*}.
\]
\end{proposition}

\begin{proof}
Let $\nu\in\R$, $\rho\ge 0$, and $\mu_i\ge 0$ be the multipliers for the budget, cash nonnegativity, and stake nonnegativity constraints. The Lagrangian is
\[
\mathcal L_0
=
\sum_{i=1}^n p_iU_\gamma(W_i)
-\nu\Bigl(c+\sum_{i=1}^n x_i-1\Bigr)
+\rho c+\sum_{i=1}^n \mu_i x_i.
\]
Stationarity with respect to $x_i$ gives
\[
\frac{p_i}{q_i}W_i^{-\gamma}-\nu+\mu_i=0,
\qquad 1\le i\le n.
\]
Hence, if $x_i>0$, then $\mu_i=0$ and
\[
L_iW_i^{-\gamma}=\nu.
\]
If $x_i=0$, then $W_i=c$ and
\[
L_ic^{-\gamma}\le \nu.
\]
By \cref{cor:overround}, the optimizer is non-full-support, so $c>0$ and therefore $\rho=0$. Stationarity with respect to $c$ gives
\[
\sum_{i=1}^n p_iW_i^{-\gamma}=\nu.
\]

The active set is an upper tail in the ordered scores $L_i$, hence a prefix $A=\{1,\dots,k\}$ for some $k\in\{1,\dots,n-1\}$. Summing the active equations with weights $q_i$ over $i\in A$ yields
\[
\sum_{i\in A}p_iW_i^{-\gamma}=\nu Q_k.
\]
Subtracting from the cash stationarity equation gives
\[
(1-P_k)c^{-\gamma}=\nu(1-Q_k),
\qquad\text{so}\qquad
\nu c^{\gamma}=\tau_k.
\]
Therefore
\[
x_i>0 \iff W_i>c \iff L_i>\nu c^\gamma=\tau_k,
\]
and inactivity is equivalent to $L_i\le \tau_k$. Thus the support condition is exactly
\[
L_k>\tau_k\ge L_{k+1}.
\]
By the uniqueness hypothesis, $k=k_*$, and therefore $A=A_*$. The active-state formula follows from
\[
W_i=\left(\frac{L_i}{\nu}\right)^{1/\gamma}
=c\left(\frac{L_i}{\tau_*}\right)^{1/\gamma},
\qquad i\in A_*.
\]
Substituting this into the budget identity
\[
(1-Q_*)c+\sum_{i\in A_*}q_iW_i=1
\]
yields the stated expression for $c_*$.
\end{proof}

\begin{corollary}[Ordinary Kelly prefix formula]
\label{cor:log-prefix}
Under the hypotheses of \cref{prop:crra-prefix}, the unconstrained logarithmic optimizer ($\gamma=1$) is
\[
W_i^{\mathrm K}=\max\{\tau_*,L_i\},
\qquad
A_*:=\{1,\dots,k_*\},
\qquad
c^{\mathrm K}=\tau_*,
\]
with stakes
\[
x_i^{\mathrm K}=
\begin{cases}
q_i(L_i-\tau_*), & i\in A_*,\\
0, & i\notin A_*.
\end{cases}
\]
\end{corollary}

\begin{proof}
Set $\gamma=1$ in \cref{prop:crra-prefix}. Then
\[
c_*=
\frac{1}{(1-Q_*)+\sum_{i\in A_*}q_iL_i/\tau_*}
=
\frac{1}{(1-Q_*)+P_*/\tau_*}
=\tau_*,
\]
using $\tau_*=(1-P_*)/(1-Q_*)$. The remaining formulas follow immediately.
\end{proof}

\section{Risk-constrained CRRA support invariance}
\label{sec:constrained-crra}

We now return to the constrained problem \eqref{eq:primal}.

\begin{theorem}[CRRA support invariance]
\label{thm:support}
Assume $Q_n>1$ and suppose the unconstrained problem satisfies the hypotheses of \cref{prop:crra-prefix}. Then every optimizer of \eqref{eq:primal} has active set
\[
A_*:=\{1,\dots,k_*\}.
\]
Equivalently, the drawdown-surrogate constraint preserves the unconstrained CRRA active set, which is itself the ordinary Kelly prefix.

More precisely, if $A:=\{i:x_i^*>0\}$ denotes the constrained active set, then there exists a scalar
\[
\tau_A=\frac{1-P_A}{1-Q_A},
\qquad
P_A:=\sum_{i\in A}p_i,
\quad
Q_A:=\sum_{i\in A}q_i,
\]
such that
\[
A=\{i:L_i>\tau_A\},
\qquad
A^c=\{i:L_i\le \tau_A\}.
\]
Hence $A$ is a prefix in likelihood-ratio order, and by uniqueness of the unconstrained prefix necessarily $A=A_*$.
\end{theorem}

\begin{proof}
By \cref{cor:overround}, every optimizer is non-full-support, hence $c^*>0$. Let $\eta\ge 0$, $\nu\in\R$, $\rho\ge 0$, and $\mu_i\ge 0$ be the KKT multipliers for the risk constraint, budget constraint, cash nonnegativity, and stake nonnegativity constraints. The Lagrangian is
\[
\mathcal L
=
\sum_{i=1}^n p_iU_\gamma(W_i)
-\eta\Bigl(\sum_{i=1}^n p_iW_i^{-\lambda}-1\Bigr)
-\nu\Bigl(c+\sum_{i=1}^n x_i-1\Bigr)
+\rho c+\sum_{i=1}^n \mu_i x_i.
\]
Stationarity with respect to $x_i$ gives
\begin{equation}
\label{eq:crra-kkt-x}
L_i\Bigl(W_i^{-\gamma}+\eta\lambda W_i^{-(\lambda+1)}\Bigr)-\nu+\mu_i=0.
\end{equation}
Since $c^*>0$, complementary slackness gives $\rho=0$, and stationarity with respect to $c$ becomes
\begin{equation}
\label{eq:crra-kkt-c}
\sum_{i=1}^n p_i\Bigl(W_i^{-\gamma}+\eta\lambda W_i^{-(\lambda+1)}\Bigr)=\nu.
\end{equation}

If $i\in A$, then $x_i^*>0$, so $\mu_i=0$ and
\begin{equation}
\label{eq:crra-active-kkt}
L_i\Bigl(W_i^{-\gamma}+\eta\lambda W_i^{-(\lambda+1)}\Bigr)=\nu.
\end{equation}
If $j\notin A$, then $W_j=c$ and \eqref{eq:crra-kkt-x} gives
\begin{equation}
\label{eq:crra-inactive-kkt}
L_j\Bigl(c^{-\gamma}+\eta\lambda c^{-(\lambda+1)}\Bigr)\le \nu.
\end{equation}
Summing \eqref{eq:crra-active-kkt} over $i\in A$ with weights $q_i$ yields
\[
\sum_{i\in A} p_i\Bigl(W_i^{-\gamma}+\eta\lambda W_i^{-(\lambda+1)}\Bigr)=\nu Q_A.
\]
Subtracting from \eqref{eq:crra-kkt-c} and using $W_j=c$ for $j\notin A$ gives
\[
(1-P_A)\Bigl(c^{-\gamma}+\eta\lambda c^{-(\lambda+1)}\Bigr)=\nu(1-Q_A).
\]
The left-hand side is strictly positive, and $\nu>0$ by \eqref{eq:crra-active-kkt}, so $1-Q_A>0$. Hence $Q_A<1$ and $\tau_A$ is well-defined and strictly positive. Therefore
\begin{equation}
\label{eq:tauA-crra}
\tau_A:=\frac{\nu}{c^{-\gamma}+\eta\lambda c^{-(\lambda+1)}}=\frac{1-P_A}{1-Q_A}.
\end{equation}

Define
\[
h_\eta(w):=w^{-\gamma}+\eta\lambda w^{-(\lambda+1)}.
\]
Then
\[
h_\eta'(w)
=-\gamma w^{-(\gamma+1)}-\eta\lambda(\lambda+1)w^{-(\lambda+2)}<0,
\qquad w>0.
\]
Thus $h_\eta$ is strictly decreasing on $(0,\infty)$. If $i\in A$, then $W_i>c$, so $h_\eta(W_i)<h_\eta(c)$. Combining this with \eqref{eq:crra-active-kkt} and \eqref{eq:tauA-crra} gives
\[
L_i>\tau_A.
\]
If $j\notin A$, then \eqref{eq:crra-inactive-kkt} and \eqref{eq:tauA-crra} yield
\[
L_j\le \tau_A.
\]
Hence
\[
A=\{i:L_i>\tau_A\},
\qquad
A^c=\{i:L_i\le \tau_A\}.
\]
Therefore $A$ is a prefix, say $A=\{1,\dots,k\}$, and its cutoff equals $\tau_k$. By the uniqueness condition \eqref{eq:prefix-condition}, we must have $k=k_*$. Thus $A=A_*$.
\end{proof}

\begin{remark}[General penalty principle]
\label{rem:general-penalty}
The proof of \cref{thm:support} uses only two structural ingredients in the non-full-support regime, where inactive states all share the common wealth level $c>0$. First, the active-state first-order condition factors as
\[
L_i\,\varphi_\eta(W_i)=\nu
\]
for a scalar statewise map $\varphi_\eta$ independent of $i$. Second, for the realized multiplier $\eta\ge 0$, the map $\varphi_\eta$ is strictly decreasing on $(0,\infty)$. Accordingly, the same threshold argument applies to any risk constraint of the form
\[
\sum_{i=1}^n p_i\psi(W_i)\le \kappa
\]
whenever the first-order equation takes the form
\[
L_i\bigl(U'(W_i)-\eta\psi'(W_i)\bigr)=\nu,
\qquad i\in A,
\]
with inactive-state inequality
\[
L_j\bigl(U'(c)-\eta\psi'(c)\bigr)\le \nu,
\qquad j\notin A,
\]
and the derivative
\[
\varphi_\eta'(w)=U''(w)-\eta\psi''(w)
\]
is strictly negative on $(0,\infty)$. In that case,
\[
W_i>c \iff \varphi_\eta(W_i)<\varphi_\eta(c)
\iff L_i>\frac{\nu}{\varphi_\eta(c)},
\]
so the support is again determined by a single likelihood-ratio threshold. We state and prove the theorem for the CRRA--power-penalty pair because this is the setting in which the subsequent logarithmic calibration theory becomes completely explicit.
\end{remark}

\section{Logarithmic utility: fixed-support calibration}
\label{sec:log}

We now specialize to $\gamma=1$. The support theorem above is stated for all CRRA parameters, but the remaining solver theory is developed only in the logarithmic case. The reason is structural rather than conceptual: once the common support is known, the general-$\gamma$ problem still reduces to a one-dimensional outer calibration coupled to independent scalar inner equations, but the logarithmic case is the one in which that reduction yields the cleanest explicit formulas and calibration theorem. By \cref{thm:support}, the constrained optimizer has the same active set as ordinary Kelly, namely
\[
A_*:=\{1,\dots,k_*\}.
\]
Set
\[
P_*:=P_{k_*},
\qquad
Q_*:=Q_{k_*},
\qquad
\tau_*:=\tau_{k_*}=\frac{1-P_*}{1-Q_*},
\qquad
r_i:=\frac{L_i}{\tau_*}>1\quad (i\in A_*).
\]

The general-CRRA reduction on this fixed support takes the form
\[
r_i\bigl(z^{-\gamma}+s z^{-(\lambda+1)}\bigr)=1+s,
\qquad
s:=\eta\lambda c^{-(\lambda+1-\gamma)}.
\]
In the logarithmic case $\gamma=1$, this becomes the cleaner equation developed below. Accordingly, we restrict attention to $\gamma=1$ for the remainder of this section and for the explicit algorithmic corollaries that follow.

\begin{lemma}[One-dimensional active-state equations]
\label{lem:activeeq}
Assume $\gamma=1$. Let $(c^*,x^*)$ be a constrained optimizer. Define
\[
s:=\eta\lambda c^{-\lambda}\ge 0,
\qquad
W_i=cz_i\quad (i\in A_*).
\]
Then each active state ratio $z_i>1$ solves
\begin{equation}
\label{eq:scalar-z}
r_i\bigl(z^{-1}+s z^{-(\lambda+1)}\bigr)=1+s,
\end{equation}
or equivalently
\begin{equation}
\label{eq:poly-z}
(1+s)z^{\lambda+1}-r_i z^\lambda-r_i s=0.
\end{equation}
Conversely, if $z_i>1$ satisfies \eqref{eq:scalar-z} for each $i\in A_*$ and the budget and risk constraints hold, then the resulting point satisfies the KKT system.
\end{lemma}

\begin{proof}
Under $\gamma=1$, equation \eqref{eq:tauA-crra} becomes
\[
\nu=\tau_*\Bigl(c^{-1}+\eta\lambda c^{-(\lambda+1)}\Bigr)=\tau_*c^{-1}(1+s).
\]
Writing $W_i=cz_i$ in the active-state equation \eqref{eq:crra-active-kkt} gives
\[
L_i\Bigl((cz_i)^{-1}+\eta\lambda(cz_i)^{-(\lambda+1)}\Bigr)=\tau_*c^{-1}(1+s).
\]
Multiplying by $c/\tau_*$ and using $r_i=L_i/\tau_*$ yields \eqref{eq:scalar-z}. Clearing denominators gives \eqref{eq:poly-z}.

Conversely, if \eqref{eq:scalar-z} holds for each $i\in A_*$, then reversing the normalization recovers the active-state KKT equations. The inactive-state inequalities follow from $L_j\le \tau_*$ and the definition of $\tau_*$. Hence the KKT system holds once the budget and risk constraints are imposed.
\end{proof}

\begin{lemma}[Existence, uniqueness, and monotonicity of the inner solves]
\label{lem:z-properties}
Assume $\gamma=1$. For each fixed $i\in A_*$ and each $s\ge 0$, equation \eqref{eq:scalar-z} has a unique solution
\[
z_i(s)\in (1,r_i].
\]
Moreover,
\begin{enumerate}[label=\textnormal{(\alph*)},leftmargin=*]
\item $s\mapsto z_i(s)$ is continuous and strictly decreasing;
\item $z_i(0)=r_i$;
\item
\[
\lim_{s\to\infty} z_i(s)=r_i^{1/(\lambda+1)}.
\]
\end{enumerate}
\end{lemma}

\begin{proof}
Fix $i\in A_*$. Define
\[
H_{r_i,s}(z):=r_i\bigl(z^{-1}+s z^{-(\lambda+1)}\bigr).
\]
For fixed $s\ge 0$, the map $z\mapsto H_{r_i,s}(z)$ is strictly decreasing on $(0,\infty)$. Moreover,
\[
H_{r_i,s}(1)=r_i(1+s)>1+s,
\qquad
H_{r_i,s}(r_i)=1+s r_i^{-\lambda}<1+s.
\]
Hence there is a unique $z_i(s)\in(1,r_i)$ satisfying \eqref{eq:scalar-z} when $s>0$, and $z_i(0)=r_i$ when $s=0$.

Let
\[
F_i(z,s):=(1+s)z^{\lambda+1}-r_i z^\lambda-r_i s.
\]
At the root $z=z_i(s)$, implicit differentiation gives
\[
z_i'(s)=-\frac{\partial_sF_i}{\partial_zF_i}.
\]
Using \eqref{eq:poly-z},
\[
\partial_sF_i=z^{\lambda+1}-r_i
=\frac{r_i(z^\lambda-1)}{1+s}>0
\]
at the root because $z_i(s)>1$. Also
\[
\partial_zF_i=(\lambda+1)(1+s)z^\lambda-\lambda r_i z^{\lambda-1}>0
\]
at the root, since $F_i(\cdot,s)$ crosses zero exactly once from below to above. Therefore $z_i'(s)<0$, and continuity follows from the implicit function theorem.

Finally, dividing \eqref{eq:scalar-z} by $s$ and letting $s\to\infty$ yields
\[
r_i z^{-(\lambda+1)}=1,
\]
so $z_i(s)\to r_i^{1/(\lambda+1)}$.
\end{proof}

For $s\ge 0$, define
\begin{equation}
\label{eq:def-B}
B(s):=1+\sum_{i\in A_*} q_i\bigl(z_i(s)-1\bigr),
\qquad
c(s):=\frac{1}{B(s)},
\end{equation}
and
\begin{equation}
\label{eq:def-R}
R(s):=c(s)^{-\lambda}\left((1-P_*)+\sum_{i\in A_*} p_i z_i(s)^{-\lambda}\right).
\end{equation}

\begin{theorem}[One-dimensional calibration theorem]
\label{thm:solver}
Assume $\gamma=1$ and the hypotheses of \cref{prop:crra-prefix}. Then:
\begin{enumerate}[label=\textnormal{(\alph*)},leftmargin=*]
\item $R:[0,\infty)\to(0,\infty)$ is continuous and strictly decreasing.
\item Its initial value is the risk functional evaluated at the ordinary Kelly solution:
\[
R(0)=\tau_*^{-\lambda}(1-P_*)+\sum_{i\in A_*} p_iL_i^{-\lambda}.
\]
\item Its limit satisfies
\[
\lim_{s\to\infty}R(s)
=\tau_*\left((1-Q_*)+\sum_{i\in A_*} q_i r_i^{1/(\lambda+1)}\right)^{\lambda+1}
< 1.
\]
\item If $R(0)\le 1$, then the ordinary Kelly solution is feasible for \eqref{eq:primal} and is therefore optimal.
\item If $R(0)>1$, then there exists a unique $s_*>0$ such that $R(s_*)=1$. The unique optimal terminal wealth vector is
\[
W_i^*=
\begin{cases}
c(s_*)z_i(s_*), & i\in A_*,\\
c(s_*), & i\notin A_*,
\end{cases}
\]
and the optimal stakes are
\[
x_i^*=
\begin{cases}
q_i c(s_*)\bigl(z_i(s_*)-1\bigr), & i\in A_*,\\
0, & i\notin A_*.
\end{cases}
\]
The associated dual parameters are
\[
\eta^*=\frac{s_*c(s_*)^\lambda}{\lambda},
\qquad
\nu^*=\tau_* c(s_*)^{-1}(1+s_*).
\]
\end{enumerate}
\end{theorem}

\begin{proof}
By \cref{lem:z-properties}, each $z_i(s)$ is continuous and strictly decreasing. Hence $B(s)$ in \eqref{eq:def-B} is continuous and strictly decreasing, so $c(s)$ is continuous and strictly increasing. This implies that $R$ is continuous.

At $s=0$, \cref{lem:z-properties} gives $z_i(0)=r_i=L_i/\tau_*$. Therefore
\[
B(0)=1+\sum_{i\in A_*}q_i\left(\frac{L_i}{\tau_*}-1\right)
=1+\frac{P_*}{\tau_*}-Q_*
=\frac{1}{\tau_*},
\]
so $c(0)=\tau_*$. Substituting into \eqref{eq:def-R} yields the stated formula for $R(0)$.

As $s\to\infty$, \cref{lem:z-properties} gives $z_i(s)\to r_i^{1/(\lambda+1)}$. Hence
\[
B(s)\to C:=(1-Q_*)+\sum_{i\in A_*} q_i r_i^{1/(\lambda+1)},
\qquad
c(s)\to C^{-1}.
\]
Also,
\[
(1-P_*)+\sum_{i\in A_*}p_i z_i(s)^{-\lambda}
\to
(1-P_*)+\sum_{i\in A_*} p_i r_i^{-\lambda/(\lambda+1)}.
\]
Since $p_i=\tau_* q_i r_i$ and $1-P_*=\tau_*(1-Q_*)$, the limit of the bracket equals $\tau_* C$. Therefore
\[
\lim_{s\to\infty}R(s)=\tau_* C^{\lambda+1}.
\]
Because the weights $1-Q_*$ and $q_i$ sum to $1$, and because the arguments $1$ and $r_i^{1/(\lambda+1)}$ are not all equal (indeed $r_i>1$ for every $i\in A_*$), strict Jensen for the strictly convex function $t\mapsto t^{\lambda+1}$ gives
\[
C^{\lambda+1}
< (1-Q_*)+\sum_{i\in A_*}q_i r_i
=\frac{1-P_*}{\tau_*}+\frac{P_*}{\tau_*}
=\frac{1}{\tau_*}.
\]
Therefore
\[
\lim_{s\to\infty}R(s)=\tau_* C^{\lambda+1}<1,
\]
which proves the strict limit bound.

To show monotonicity, fix the support $A_*$ and define
\[
\Psi(c,W):=(1-P_*)c^{-\lambda}+\sum_{i\in A_*}p_iW_i^{-\lambda},
\]
\[
f(c,W):=(1-P_*)\log c+\sum_{i\in A_*}p_i\log W_i.
\]
For each $\eta\ge 0$, consider the strictly concave penalized problem
\[
\max\Bigl\{f(c,W)-\eta\bigl(\Psi(c,W)-1\bigr)\Bigr\}
\]
subject to
\[
(1-Q_*)c+\sum_{i\in A_*}q_iW_i=1,
\qquad
W_i\ge c.
\]
Uniqueness of the penalized optimizer implies the standard monotonicity comparison: if $0\le \eta_1<\eta_2$ and $(c_m,W^{(m)})$ denotes the penalized optimizer at $\eta_m$, then
\[
(\eta_2-\eta_1)\bigl(\Psi(c_1,W^{(1)})-\Psi(c_2,W^{(2)})\bigr)\ge 0.
\]
Hence the optimal risk value is nonincreasing in $\eta$. Since
\[
\eta(s)=\frac{s c(s)^\lambda}{\lambda}
\]
and $c(s)$ is strictly increasing, the map $s\mapsto\eta(s)$ is strictly increasing, so $R(s)$ is nonincreasing.

Suppose now that $0\le s_1<s_2$ but $R(s_1)=R(s_2)$. Then the previous comparison is tight, so the penalized optimizer is the same point for $\eta_1:=\eta(s_1)$ and $\eta_2:=\eta(s_2)$. Subtracting the active-state first-order equations at that common optimizer gives
\[
\lambda(\eta_2-\eta_1)p_iW_i^{-(\lambda+1)}=(\nu_2-\nu_1)q_i,
\qquad i\in A_*.
\]
Subtracting the corresponding cash equations yields
\[
\lambda(\eta_2-\eta_1)(1-P_*)c^{-(\lambda+1)}=(\nu_2-\nu_1)(1-Q_*).
\]
Eliminating $\nu_2-\nu_1$ and using $\tau_*=(1-P_*)/(1-Q_*)$ gives
\[
L_iW_i^{-(\lambda+1)}=\tau_* c^{-(\lambda+1)},
\]
so
\[
\frac{W_i}{c}=\left(\frac{L_i}{\tau_*}\right)^{1/(\lambda+1)}=r_i^{1/(\lambda+1)}.
\]
Substituting $z_i=r_i^{1/(\lambda+1)}$ into \eqref{eq:scalar-z} yields
\[
r_i^{\lambda/(\lambda+1)}+s=1+s,
\]
which is impossible because $r_i>1$. Thus $R$ is strictly decreasing.

If $R(0)\le 1$, then the ordinary Kelly solution is feasible. Since the constraint only shrinks the feasible set, that unconstrained optimum remains optimal. If $R(0)>1$, continuity, strict monotonicity, and the strict limit bound imply the existence of a unique $s_*>0$ such that $R(s_*)=1$. The reconstruction formulas follow directly from the definitions of $c(s)$, $z_i(s)$, and $s$.
\end{proof}

\begin{corollary}[Structured solver for logarithmic utility]
\label{cor:algorithm}
Assume $\gamma=1$ and the hypotheses of \cref{prop:crra-prefix}. Then the constrained problem is solved by the following procedure:
\begin{enumerate}[label=\textnormal{(\arabic*)},leftmargin=*]
\item sort the likelihood-ratio scores $L_i=p_i/q_i$ in decreasing order;
\item compute the unique prefix index $k_*$ satisfying \eqref{eq:prefix-condition};
\item evaluate
\[
R(0)=\tau_*^{-\lambda}(1-P_*)+\sum_{i\in A_*} p_iL_i^{-\lambda};
\]
\item if $R(0)\le 1$, return the ordinary Kelly solution; otherwise solve the scalar equation $R(s)=1$ for $s>0$ and reconstruct $(c^*,W^*,x^*)$ from \cref{thm:solver}.
\end{enumerate}
The complexity is $O(n\log n)$ for the initial sort, followed by a one-dimensional outer root search whose function evaluations require only independent one-dimensional inner solves on the active prefix.
\end{corollary}

\begin{proof}
This is an immediate restatement of \cref{thm:solver}.
\end{proof}

\begin{corollary}[The case $\lambda=1$]
\label{cor:lambda1}
Assume $\gamma=1$ and $\lambda=1$. Then
\[
R(0)=1.
\]
Consequently the ordinary Kelly solution is automatically feasible and therefore optimal.

In this case the active-state equation reduces to the quadratic
\[
(1+s)z^2-r_i z-r_i s=0,
\]
whose unique admissible root is
\[
z_i(s)=\frac{r_i+\sqrt{r_i^2+4r_i s(1+s)}}{2(1+s)}.
\]
\end{corollary}

\begin{proof}
For $\lambda=1$, the formula in \cref{thm:solver}(b) becomes
\[
R(0)=\tau_*^{-1}(1-P_*)+\sum_{i\in A_*} p_iL_i^{-1}
=(1-Q_*)+\sum_{i\in A_*} q_i=1.
\]
Hence the first claim follows from \cref{thm:solver}(d).

The quadratic formula follows by setting $\lambda=1$ in \eqref{eq:poly-z}. The positive root is the displayed one, and it exceeds $1$ because $r_i>1$.
\end{proof}

\section{Full-support regimes}
\label{sec:full-support}

The support theory developed in the previous sections is an overround non-full-support theory. The present section records only benchmark observations for the complementary fair and subfair regimes; it does not provide a parallel exact support-and-calibration theory of comparable sharpness. By \cref{prop:full-support-obstruction}, the optimizer automatically has an inactive set when $Q_n>1$. When $Q_n\le 1$, by contrast, full-support behavior can occur.

\begin{proposition}[The fair logarithmic benchmark]
\label{prop:fair-log}
Assume $\gamma=1$ and $\sum_{i=1}^n q_i=1$.
\begin{enumerate}[label=\textnormal{(\alph*)},leftmargin=*]
\item The unconstrained logarithmic optimizer has full support and terminal wealth vector
\[
W_i^{\mathrm K}=L_i=\frac{p_i}{q_i},\qquad 1\le i\le n.
\]
\item If $0<\lambda\le 1$, then this full-support Kelly point satisfies the risk constraint and is therefore also the constrained optimizer.
\item If $\lambda>1$ and a constrained optimizer has full support, then it satisfies the interior KKT system
\[
L_i\Bigl(W_i^{-1}+\eta\lambda W_i^{-(\lambda+1)}\Bigr)=\nu,
\qquad 1\le i\le n,
\]
with
\[
\sum_{i=1}^n q_iW_i=1,
\qquad
\sum_{i=1}^n p_iW_i^{-\lambda}=1
\]
whenever $\eta>0$. In particular, the optimal wealth vector is strictly increasing in the likelihood-ratio scores $L_i$.
\end{enumerate}
\end{proposition}

\begin{proof}
When $\sum_i q_i=1$, the unconstrained logarithmic first-order conditions are
\[
\frac{p_i}{q_iW_i}=\nu,
\qquad 1\le i\le n.
\]
Thus $W_i=L_i/\nu$, and the budget identity $\sum_i q_iW_i=1$ forces $\nu=1$, giving $W_i^{\mathrm K}=L_i$.

At this point,
\[
\sum_{i=1}^n p_i(L_i)^{-\lambda}=
\sum_{i=1}^n p_i^{1-\lambda}q_i^{\lambda}.
\]
For $0<\lambda\le 1$, weighted AM--GM gives
\[
p_i^{1-\lambda}q_i^{\lambda}\le (1-\lambda)p_i+\lambda q_i.
\]
Summing over $i$ yields
\[
\sum_{i=1}^n p_i^{1-\lambda}q_i^{\lambda}
\le (1-\lambda)\sum_{i=1}^n p_i+\lambda\sum_{i=1}^n q_i=1,
\]
so the risk constraint is satisfied and the unconstrained optimum remains optimal.

For $\lambda>1$, the displayed interior KKT system is simply the full-support specialization of the constrained first-order conditions. The map
\[
w\mapsto w^{-1}+\eta\lambda w^{-(\lambda+1)}
\]
is strictly decreasing, so the solution must be strictly increasing in $L_i$.
\end{proof}

\begin{remark}[Subfair markets]
If $\sum_i q_i<1$, then \cref{prop:full-support-obstruction} shows that a full-support optimizer is not excluded. In that case the cash multiplier is strictly positive and any full-support optimizer must have $c=0$. This regime is atypical for bookmaker markets and plays no role in the overround support theory developed above.
\end{remark}

\section{Numerical example}
\label{sec:example}

The next example illustrates the nontrivial logarithmic calibration when $\lambda>1$.

\begin{example}[Three outcomes, one active claim]
Take
\[
(p_1,p_2,p_3)=(0.50,0.30,0.20),
\qquad
(q_1,q_2,q_3)=(0.45,0.35,0.30),
\qquad
\gamma=1,
\qquad
\lambda=2.
\]
Then $Q_n=1.10>1$, so we are in the overround regime. The likelihood-ratio scores are
\[
L=(1.111\ldots,0.857\ldots,0.667\ldots).
\]
Moreover,
\[
\tau_1=\frac{1-p_1}{1-q_1}=\frac{0.50}{0.55}=0.909\ldots,
\]
so
\[
L_1>\tau_1\ge L_2.
\]
Hence $k_*=1$ and the common unconstrained/constrained active set is $A_* = \{1\}$.

The unconstrained Kelly point is
\[
c^{\mathrm K}=\tau_1=0.9091,
\qquad
x_1^{\mathrm K}=q_1(L_1-\tau_1)=0.0909,
\qquad
x_2^{\mathrm K}=x_3^{\mathrm K}=0.
\]
The initial value of the calibration functional is
\[
R(0)=\tau_1^{-2}(1-p_1)+p_1L_1^{-2}=1.01>1,
\]
so the risk constraint is binding and the unconstrained Kelly point is no longer feasible.

Solving the scalar equation $R(s)=1$ gives
\[
s_*=0.3794\ \text{(approximately)}.
\]
The corresponding inner solution is
\[
z_1(s_*)=1.1433,
\qquad
c^*=0.9394,
\qquad
W_1^*=c^*z_1(s_*)=1.0740,
\qquad
x_1^*=q_1c^*(z_1(s_*)-1)=0.0606.
\]
Thus the support remains $\{1\}$, exactly as predicted by \cref{thm:support}, but the funded wealth profile moves toward the cash floor:
\[
(c^{\mathrm K},W_1^{\mathrm K})=(0.9091,1.1111)
\quad\longrightarrow\quad
(c^*,W_1^*)=(0.9394,1.0740).
\]
The risk constraint therefore acts through multiplier calibration on a fixed support rather than by changing the active set.
\end{example}

\section{Conclusion}

The state-price/wealth-profile formulation turns the mutually exclusive risk-constrained Kelly problem into an exact support-and-calibration problem. In the overround regime, every optimizer is automatically non-full-support. In that regime, the unconstrained CRRA optimizer has the same likelihood-ratio prefix support as ordinary Kelly, and the drawdown-surrogate constraint preserves that support. The constraint changes only the funded wealth profile.

For logarithmic utility, this structural theorem yields a complete one-dimensional calibration theory: once the common prefix is known, the constrained optimizer is recovered from a single scalar outer equation together with independent one-dimensional inner solves. This is the precise sense in which the mutually exclusive state-price problem is much more explicit than a generic finite-outcome convex program would suggest.

\end{document}